\documentclass[12pt]{amsart}
\usepackage{amssymb,amstext,amsmath,amscd,amsthm,amsfonts,enumerate,graphicx,latexsym,color}
\def\latex/{{\protect\LaTeX}}
\def\latexe/{{\protect\LaTeXe}}
\def\amslatex/{{\protect\AmS-\protect\LaTeX}}
\def\tex/{{\protect\TeX}}
\def\amstex/{{\protect\AmS-\protect\TeX}}
\def\bibtex/{{Bib\protect\TeX}}
\def\makeindx/{\textit{MakeIndex}}
\theoremstyle{plain} %default
\newtheorem{thm}{Theorem}[section]
\newtheorem{prop}[thm]{Proposition}

\newtheorem{cor}[thm]{Corollary}

\theoremstyle{definition}

\newtheorem{dfn}[thm]{Definition}

\newtheorem{eg}[thm]{Example}
\newtheorem{conj}[thm]{Conjecture}
\newtheorem{ques}[thm]{Question}
\newtheorem{rmk}[thm]{Remark}

\theoremstyle{remark}
\newtheorem*{ac}{Acknowledgments}
\newcommand{\CC}{\mathbb{C}}

\DeclareMathOperator{\id}{id}

\DeclareMathOperator{\Tor}{Tor}
\DeclareMathOperator{\Ext}{Ext}
\DeclareMathOperator{\Hom}{Hom}
\DeclareMathOperator{\Ho}{H}
\DeclareMathOperator{\CI}{CI-dim}
\DeclareMathOperator{\G-dim}{G-dim}

\DeclareMathOperator{\GC-dim}{G_\text{$C$}-dim}

 \DeclareMathOperator{\Supp}{Supp}
 \DeclareMathOperator{\Spec}{Spec}

 \DeclareMathOperator{\pd}{pd}
 \DeclareMathOperator{\fd}{fd}

 \DeclareMathOperator{\cx}{cx}

 \DeclareMathOperator{\cod}{codim}
 \DeclareMathOperator{\depth}{depth}
 \DeclareMathOperator{\coker}{coker}

\DeclareMathOperator{\RHom}{{\bf R}Hom}
\newcommand{\tensor}{\otimes^{\bf L}}

 \DeclareMathOperator{\syz}{\mathrm{\Omega}}

\DeclareMathOperator{\NT}{\operatorname{\mathsf{NT}}}
\DeclareMathOperator{\res}{\operatorname{\mathsf{res}}}

\DeclareMathOperator{\md}{\operatorname{\mathsf{mod}}}
\DeclareMathOperator{\fpd}{\operatorname{\mathsf{fpd}}}
\DeclareMathOperator{\CM}{\operatorname{\mathsf{CM}}}
\DeclareMathOperator{\ttp}{\operatorname{\mathsf{T}}}
\DeclareMathOperator{\tei}{\operatorname{\mathsf{EI}}}
\DeclareMathOperator{\tep}{\operatorname{\mathsf{EP}}}

\def\mod{\operatorname{mod}}
\begin{document}
\baselineskip=15pt
\title{Modules that detect finite homological dimensions}
\author{Olgur Celikbas}
\address{Department of Mathematics, 323 Mathematical Sciences Bldg, University of Missouri--Columbia, Columbia, MO 65211 USA}
\email{celikbaso@missouri.edu}
\urladdr{http://www.math.missouri.edu/~celikbaso/}
\author{Hailong Dao}
\address{Department of Mathematics, University of Kansas, 405 Snow Hall, 1460 Jayhawk
Blvd, Lawrence, KS 66045-7523, USA}
\email{hdao@math.ku.edu}
\urladdr{http://www.math.ku.edu/~hdao/}
\author{Ryo Takahashi}
\address{Graduate School of Mathematics, Nagoya University, Furocho, Chikusaku, Nagoya 464-8602, Japan/Department of Mathematics, University of Nebraska, Lincoln, NE 68588-0130, USA}
\email{takahashi@math.nagoya-u.ac.jp}
\urladdr{http://www.math.nagoya-u.ac.jp/~takahashi/}
\thanks{2010 {\em Mathematics Subject Classification.} Primary 13D07; Secondary 13D05, 13H10}
\thanks{{\em Key words and phrases.} test module, homological dimension, vanishing of Tor and Ext}
\thanks{The second author was partially supported by NSF grants DMS 0834050 and DMS 1104017. The third author was partially supported by JSPS Grant-in-Aid for Young Scientists (B) 22740008 and by JSPS Postdoctoral Fellowships for Research Abroad}
\begin{abstract}
We study homological properties of test modules that are, in principle, modules that detect finite homological dimensions.
The main outcome of our results is a generalization of a classical theorem of Auslander and Bridger: we prove that, if a commutative Noetherian complete local ring $R$ admits a test module of finite Gorenstein dimension, then $R$ is Gorenstein.
\end{abstract}
\maketitle
%%%%%%%%%%%%%%%%%%%%%%%%%%%%%%%%%%%%%%%%%%%%%%%%%%%%%%%%%%%%%%%%%%%%
\section{Introduction}

Throughout this paper, we assume that all rings are commutative Noetherian rings and all modules are finitely generated.
Unless otherwise specified, $R$ denotes a local ring with maximal ideal $\mathfrak{m}$ and residue field $k$.
The aim of this paper is to study test modules.

\begin{dfn}
An $R$-module $M$ is called a {\em test module} if all $R$-modules $N$ with $\Tor_{\gg 0}^{R}(M,N)=0$ have finite projective dimension.
\end{dfn}

There are many examples of test modules with interesting consequences in the literature.
For instance, it is well-known that the residue field $k$ of $R$ is a test module \cite[1.3]{BH}.
In general it requires highly nontrivial results to characterize all test modules, even over specific rings.
One such result is due to Huneke and Wiegand \cite[1.9]{HW2}: a test module over a singular hypersurface is nothing but a module of infinite projective dimension.
This result was later obtained by Miller \cite[1.1]{Mi} and the third author \cite[7.2]{Ryo2} by using different techniques. 

In this paper we will consider test modules discussed above in a broader context by studying their homological properties.
We will investigate when a module-finite algebra is a test module, and prove the following.

\begin{thm}\label{b}
Let $R\to S$ be a finite local homomorphism of local rings.
Assume either
\begin{enumerate}[\rm(1)]
\item
The ring $S$ is regular, or
\item
There exists a test $S$-module that has finite projective dimension over $R$.
\end{enumerate}
Then $S$ is a test $R$-module.
\end{thm}

%\noindent
%In particular, under the assumption (2), if $S$ is a quotient ring of $R$ by a regular sequence, then $R$ is regular, and hence $S$ is a complete intersection.

It is remarkable that the existence of a test module of finite homological dimension characterizes the ring itself: if there exists a test module of finite projective dimension, then $R$ is regular \cite[2.2.7]{BH}.
Auslander and Bridger \cite{AuBr} proved that if the residue field $k$ has finite \emph{G-dimension} (Definition \ref{GCdim}), then $R$ is Gorenstein.
Corso--Huneke--Katz--Vasconcelos and Goto--Hayasaka, under mild technical conditions, generalized this classical theorem: let $I$ be an integrally closed $\mathfrak{m}$-primary ideal of a local ring $R$.
Then $I$ is a test module \cite[3.3]{CHKV}.
If $I$ has finite G-dimension and contains a non zero-divisor of $R$, then $R$ is Gorenstein \cite[1.1(2)]{GH}.
We will obtain a generalization in this direction:

\begin{thm}\label{c}
Let $R$ be a homomorphic image of a Gorenstein local ring.
If there exists a test module of finite G-dimension, then $R$ is Gorenstein.
\end{thm}

We will also study the structure of test and nontest modules over complete intersections.
Recall that $R$ is called a \emph{complete intersection} (respectively, \emph{hypersurface}) if its completion is a quotient of a regular local ring by a regular sequence (respectively, regular element).
The notion of a \emph{resolving subcategory} of the category $\mod R$ of finitely generated $R$-modules has been introduced by Auslander and Bridger \cite{AuBr}, which is a full subcategory containing free modules and closed under direct summands, extensions and syzygies.
For each $R$-module $M$ the smallest resolving subcategory containing $M$ is called the \emph{resolving closure} of $M$.
We will prove the following.

\begin{thm}\label{a}
Let $R$ be a complete intersection.
\begin{enumerate}[\rm(1)]
\item
The test $R$-modules are precisely the $R$-modules of maximal complexity.
\item
The nontest $R$-modules form a resolving subcategory of $\mod R$.
If it is written as the resolving closure of some module, then $R$ is a hypersurface.
\end{enumerate}
\end{thm}

\noindent
The first assertion extends the result of Huneke and Wiegand stated above.
The second says nontest modules form a good subcategory but its structure is not simple in general.

The organization of this paper is as follows.
In Section 2, we analyze basic properties of test modules.
Theorem \ref{b}(1) and an extended version of Theorem \ref{b}(2) are shown in this section (Proposition \ref{1036} and Theorem \ref{thmTest}).
We also prove Theorem \ref{a}(1) in this section (Proposition \ref{propCI}).
In Section 3, we study the existence of test modules of finite homological dimensions.
We characterize test modules in terms of vanishing of Ext, which yields a generalized version of Theorem \ref{c} (Theorem \ref{mainthm1}).
In Section 4, we develop categorical approaches for nontest modules.
Theorem \ref{a}(2) is proved in this section (Corollaries \ref{12241} and \ref{cicor}).

%%%%%%%%%%%%%%%%%%%%%%%%%%%%%%%%%%%%%%%%%%%%%%%%%%%%%%%%%%%%%%%%%%%%%%%
\section{Basic properties of test modules}

In this section we analyze basic properties of test modules.
We should note that modules akin to test modules were studied in the literature before, see for example \cite{IP} and \cite{Levin}. 

First of all, we remark that our definition of a test module is \emph{different} from the one defined by Ramras \cite[1.1]{Ramras} (see also \cite{perp} and \cite{Jot}).
He defined and studied test modules for projectivity in terms of the vanishing of a single Ext module.
More precisely, an $R$-module $M$ is called an \textit{Ext-test} module (a test module in the sense of \cite{Jot} and \cite{Ramras}) if every $R$-module $P$ with $\Ext^{1}_{R}(P,M)=0$ is free.
We record a few observations concerned with the test and Ext-test modules:

\begin{rmk}
(1) Nontrivial examples of test modules over arbitrary local rings are abundant: if $(R, \mathfrak{m})$ is a local ring and $M\in \md R$, then there exists an integer $n>0$ such that $\mathfrak{m}^{n}M$ is a test module, see \cite[1.5(1), 2.1, 2.2 and 2.3(5)]{Tony} and \cite[2.4(b)]{Yas}.\\
(2) Test and Ext-test modules are different in general: by definition an Ext-test module has depth at most one, but a test module does not necessarily so.\\
(3) Ext-test modules are indeed test modules over complete intersections: this follows from the fact that the vanishing of $\Ext^{\gg0}_{}(-,-)$ is equivalent to the vanishing of $\Tor^{}_{\gg 0}(-,-)$ over complete intersections by \cite[6.1]{AvBu}.\\
(4) Test modules are indeed Ext-test modules over hypersurfaces that are either Artinian rings or one-dimensional domains: asssume $R$ is such a ring, $M$ is a test $R$-module and $\Ext^{1}_{R}(P,M)=0$ for some $P\in \md R$.
We may assume by \cite[Theorem 1]{Jot} that $R$ is nonregular, whence $\pd_RM=\infty$.
We see $\Ext^{>0}_{R}(P,M)=0$ from \cite[3.5]{Be2} and \cite[4.14]{CeD}.
This implies that $P$ is free by \cite[5.12]{AvBu}. 
\end{rmk}

Next we will prove that test modules behave well modulo non zero-divisors.
We denote by $\ttp(R)$ the full subcategory of $\md R$ consisting of test modules, and by $\Omega^nM$ (or $\Omega^{n}_RM$ when necessary) the $n$th syzygy of an $R$-module $M$.

\begin{prop}\label{moduloNZD}
Let $(R, \mathfrak{m})$ be a local ring and $M$ an $R$-module.
Let $x\in \mathfrak{m}$ be a non zero-divisor on $M$.
Then:
\begin{enumerate}[\rm(i)]
\item $M\in \ttp(R)$ if and only if $M/xM\in \ttp(R)$.
\item  Assume further that $x$ is a non zero-divisor on $R$.
\begin{enumerate}[\rm(a)]
\item If $M/xM\in \ttp(R/xR)$, then $M\in \ttp(R)$. 
\item If $x\notin \mathfrak{m}^{2}$ and $M\in \ttp(R)$, then $M/xM\in \ttp(R/xR)$.
\end{enumerate}
\end{enumerate}
\end{prop}

\begin{proof}
There is an exact sequence $0 \to M \xrightarrow{x} M \to M/xM \to 0$.
Hence $\Tor_{\gg 0}^{R}(M,N)=0$ if and only if $\Tor_{\gg 0}^{R}(M/xM,N)=0$ for all $R$-modules $N$.
This proves (i).
For the rest of the proof, we assume that $x$ is a non zero-divisor on both $R$ and $M$.
Assume that $M/xM\in \ttp(R/xR)$ and that $\Tor_{\gg 0}^{R}(M,N)=0$ for some $R$-module $N$.
Then, by the above exact sequence, $\Tor_{\gg 0}^{R}(M/xM,N')=0$, where $N':=\Omega N$.
Then, since $x$ is a non zero-divisor on both $R$ and $N'$, it follows that $\Tor_{\gg 0}^{R/xR}(M/xM,N'/xN')=0$.
Therefore $\pd_{R/xR}(N'/xN')<\infty$.
As $\pd_{R}(N')=\pd_{R/xR}(N'/xN')$ by \cite[1.3.5]{BH}, this proves (ii)(a).
Finally assume $x\notin \mathfrak{m}^{2}$ and $M\in \ttp(R)$.
Suppose $\Tor_{\gg 0}^{R/xR}(M/xM,T)=0$ for some $R/xR$-module $T$.
Then, as $\Tor_{i}^{R/xR}(M/xM,T)\cong \Tor_{i}^{R}(M,T)$ for all $i\geq 0$, we have $\pd_RT<\infty$.
Since $x\notin \mathfrak{m}^{2}$, we have $\pd_{R/xR}T<\infty$ by \cite[3.3.5(1)]{Av1}.
This proves that $M/xM\in \ttp(R/xR)$ and hence finishes the proof.
\end{proof}

\begin{rmk}
The assumption that $x\notin \mathfrak{m}^{2}$ in Proposition \ref{moduloNZD}(ii)(b) is necessary: assume $(R, \mathfrak{m})$ is a regular local ring and $0\neq x\in \mathfrak{m}^{2}$. Set $S=R/xR$ and $M=R$. Then $M\in \md R=\ttp(R)$. However, since $S$ is not regular, $M/xM\notin \ttp(R/xR)$.
\end{rmk}

Recall that a local homomorphism $f:R\to S$ of local rings is called a \emph{finite local ring homomorphism} if $S$ is a finitely generated $R$-module via $f$.
Now we study the behavior of test modules under finite local ring homomorphisms.
First, we point out that regular extensions of local rings are test modules.

\begin{prop}\label{1036}
Let $(R, \mathfrak{m},k) \to (S, \mathfrak{n},l)$ be a finite local ring homomorphism.
If $S$ is regular, then $S\in \ttp(R)$.
\end{prop}

\begin{proof}
Suppose that $\Tor_{\gg 0}^{R}(M,S)=0$ for some $R$-module $M$.
There is an exact sequence $0 \to G_{d} \to G_{d-1} \to \ldots \to G_{0} \to S/\mathfrak{m}S \to 0$ of $S$-modules with $G_{i}$ being free.
We see from this that $\Tor_{\gg 0}^{R}(M,S/\mathfrak{m}S)=0$.
Since $S/\mathfrak{m}S=k^{(n)}$ for some $n>0$, we get $\Tor_{\gg 0}^{R}(M,k)=0$.
This implies that $M$ has finite projective dimension and hence $S\in \ttp(R)$.
\end{proof}

\begin{thm}\label{thmTest}
Let $R\to S$ be a finite local ring homomorphism and $M\in \ttp(S)$.
Assume any $R$-module $X$ with $\Tor_{\gg 0}^{R}(S,X)=0$ satisfies $\Tor_{\gg 0}^{R}(M,X)=0$.
Then $S\in \ttp(R)$.
\end{thm}

\begin{proof}
Assume on the contrary that $S\notin \ttp(R)$.
Then there exists an $R$-module $L$ such that $\Tor_{\gg 0}^{R}(S,L)=0$ and $\pd_RL=\infty$.
Hence $\Tor_{>0}^{R}(S,T)=0$ for some $T=\Omega_{R}^nL$.
Set $d=\depth S$ and let $N=\syz_{R}^{d}T$.
There is an exact sequence $0 \to N \to F_{d-1} \to \ldots \to F_{0} \to T \to 0$ of $R$-modules, where each $F_{i}$ is a free $R$-module.
We deduce that
$$
0 \to N\otimes_{R}S \to F_{d-1}\otimes_{R}S \to \ldots \to F_{0}\otimes_{R}S \to T\otimes_{R}S \to 0
$$
is exact.
This implies $\depth_{S}(N\otimes_{R}S)\geq d$.
Using \cite[A.4.20]{Larsbook}, we have isomorphisms
$$
N\tensor_{R}M\simeq N\tensor_{R}(S\tensor_{S}M)\simeq (N\tensor_{R}S)\tensor_{S}M\simeq (N\otimes_{R}S)\tensor_{S}M
$$
in the derived category of $R$, whose last isomorphism holds by $\Tor_{>0}^{R}(S,N)=0$.
Thus $\Tor_{i}^{R}(M,N)\cong \Tor_{i}^{S}(M,N\otimes_{R}S)$ for all $i>0$.
By assumption, we have $\Tor_{\gg 0}^{R}(M,N)=0$, whence $\Tor_{\gg0}^S(M,N\otimes_RS)=0$.
As $M\in \ttp(S)$, it holds that $\pd_S(N\otimes_{R}S)<\infty$.
The Auslander-Buchsbaum formula shows that $N\otimes_{R}S$ is a free $S$-module.
Let $G$ be a minimal free resolution of $N$ over $R$.
Then $G\otimes_{R}S$ is a minimal free resolution of $N\otimes_{R}S$ over $S$.
The uniqueness of minimal free resolutions implies $G_{i}\otimes_{R}S=0$, and hence $G_i=0$, for all $i>0$.
Therefore $N$ is a free $R$-module, and $\pd_RL<\infty$.
This is a contradiction.
\end{proof}

We record a direct consequence of Theorem \ref{thmTest}.

\begin{cor}\label{propRegularextension}
Let $R\to S$ be a finite local ring homomorphism. 
If there is $M\in\ttp(S)$ with $\pd_RM<\infty$, then $S\in \ttp(R)$.
In particular, if $S=R/(\underline{x})$ where $\underline{x}$ is a regular sequence on $R$, then the ring $R$ is regular.
\end{cor}

The category of test modules over complete intersection rings is determined in terms of \emph{complexity}.
Recall that the complexity $\cx_{R}M$ of an $R$-module $M$ is the dimension of the \emph{support variety} $V(M)$ associated to $M$; see \cite{Av1} and \cite{AvBu} for details.
In an earlier version of this article, we made use of Corollary \ref{propRegularextension} and \cite[1.3]{Jo1}, and proved Proposition \ref{propCI} below for complete intersection local rings which are \emph{complete}.
The authors are grateful to Petter Andreas Bergh for explaining them a \emph{completion free} proof of this fact.

\begin{prop}\label{propCI} 
Let $R$ be a local complete intersection.
Then
$$
\ttp(R)= \{M\in \md R\mid\cx_{R}M=\cod R\}.
$$
\end{prop}

\begin{proof}
($\supseteq$):
This follows from \cite[1.2]{Jo2}.

($\subseteq$):
Put $c=\cod R$.
Let $M\in \md R$ with $\cx_{R}(M)<c$.
Then $\dim V(M)<c=\dim \overline{k}^c$, where $\overline{k}$ is the algebraic closure of $k$.
This implies that there exists a closed homogeneous variety $W$ in $\overline{k}^c$ with $\dim W > 0$ and $W \cap V (M) = \{ 0 \}$.
Now, by \cite[2.3]{BerSV}, there exists $N\in \md R$ with $V(N)=W$.
Since $\cx_{R}N = \dim V (N)=\dim W>0$, the $R$-module $N$ has infinite projective dimension.
Recall that $V (N) \cap V (M)= \{ 0 \}$; thus we deduce from \cite[Theorem IV]{AvBu} that $\Tor^R_{\gg 0}(M,N) =0$.
Consequently $M\notin \ttp(R)$.
\end{proof}

Here are some consequences of Proposition \ref{propCI}; the first one is the result of Huneke and Wiegand \cite[1.9]{HW2} discussed in the introduction. 

\begin{cor} \label{cor2propCI}
\begin{enumerate}[\rm(i)]
\item
{\rm(Huneke--Wiegand)}
Let $R$ be a hypersurface.
Then an $R$-module $M$ is in $\ttp(R)$ if and only if $M$ has infinite projective dimension.
\item
Let $R \to S$ be a finite local ring homomorphism of local complete intersections.
If there exists an $S$-module $M$ such that $\cx_{R}M=0$ and $\cx_{S}M=\cod S$, then $\cx_{R}S = \cod R$.
In other words, if $M$ has minimal complexity over $R$ and has maximal complexity over $S$, then $S$ has maximal complexity over $R$.
\end{enumerate}
\end{cor}

\begin{proof}
Only the second claim in (ii) requires a proof.
By Proposition \ref{propCI}, $M\in \ttp(S)$.
Corollary \ref{propRegularextension} implies $S\in \ttp(R)$.
Again by Proposition \ref{propCI}, we have $\cx_{R}S = \cod R$. 
\end{proof}

\begin{rmk}
(1) Local rings over which all modules of infinite projective dimension are test modules are \emph{not} necessarily hypersurfaces.
A natural example of such a ring is a Golod ring that is not Gorenstein (e.g. $\CC[[t^3,t^4,t^5]]$).\\
(2) Test modules do not behave well under localization.
Let $(R, \mathfrak{m})$ be a $2$-dimensional local hypersurface such that $R_{p}$ is not regular for some $p\in \Spec R\setminus\{\mathfrak{m}\}$ (e.g. $R=\CC[[x,y,z]]/(xy)$ and $p=(x,y)R$).
Let $M=\Omega^{2}_{R}k$.
Then $M\in \ttp(R)$ by Corollary \ref{cor2propCI}(i).
However, $M_{p}$ is free over $R_{p}$ and hence $M_{p} \notin \ttp(R_{p})$.\\
(3) Nontest modules do not behave well under localization.
Let $(R, \mathfrak{m})$ be a complete intersection of codimension $2$ such that $R_{p}$ is regular for all $p\in \Spec R\setminus\{\mathfrak{m}\}$ (e.g. $R=\CC[[x, y, z, v]]/(x^2+y^2+z^2+v^2,x^3+y^3+z^3+v^3)$).
Let $M\in \md R$ with $\cx_{R}M=1$ (see \cite[5.7]{AGP}).
Then $M\notin \ttp(R)$ by Proposition \ref{propCI}.
However, $M_{p}\in \ttp(R_{p})$ for all $p\in \Spec R\setminus\{\mathfrak{m}\}$.
\end{rmk}

\section{Homological dimensions of test modules}

In this section we study the existence of test modules of finite homological dimensions.
We start by recalling ${\rm G}_{C}$-dimension; it is a homological invariant for modules, originally introduced by Golod \cite{Golod}, associated to a fixed semidualizing module $C$.

\begin{dfn}\label{GCdim}
Let $C$ be a \emph{semidualizing} $R$-module, i.e., an $R$-module $C$ such that the natural homomorphism $R \to \Hom_{R}(C,C)$ is an isomorphism and $\Ext^{>0}_{R}(C,C)=0$.
An $R$-module $X$ is called \emph{totally $C$-reflexive} if the natural map $X \to \Hom_R(\Hom_R(X,C),C)$ is an isomorphism and $\Ext_{R}^{>0}(X,C)=\Ext_{R}^{>0}(\Hom_R(X,C),C)=0$.
The \emph{${\rm G}_{C}$-dimension} of an $R$-module $M$, denoted by $\GC-dim_RM$, is defined as the infimum of the integers $n\ge0$ such that there exists an exact sequence $0 \to X_{n} \to \cdots  \to X_{0} \to M\to0$, where each $X_{i}$ is totally $C$-reflexive.
\end{dfn}

A totally $R$-reflexive module is simply called \emph{totally reflexive}.
The $\text{G}_{R}$-dimension of $M$ is nothing but \emph{Gorenstein dimension} (\emph{G-dimension} for short) introduced by Auslander and Bridger \cite{AuBr}, and simply denoted by $\G-dim_RM$.
A lot of studies on G-dimension have been done so far.
The details are stated in the book \cite{Larsbook} and the survey article \cite{CFH}.

Corso, Huneke, Katz and Vasconcelos \cite[3.3]{CHKV} proved that integrally closed $\mathfrak{m}$-primary ideals can be used to test for finite projective dimension.
More precisely, they proved that, if $(R, \mathfrak{m})$ is a local ring and $I$ is an an integrally closed $\mathfrak{m}$-primary ideal of $R$, then $\Tor_{n}^{R}(R/I,N)=0$ if and only if $\pd_{R}N<n$.
Goto and Hayasaka \cite[1.1]{GH} proved that, if such an ideal $I$ contains a non zero-divisor of $R$ and $\G-dim_RI<\infty$, then $R$ is Gorenstein.
Thus, integrally closed $\mathfrak{m}$-primary ideals are test modules, and existence of such ideals having finite G-dimension and positive grade forces the ring to be Gorenstein.

The main purpose of this section is to generalize this.
More precisely, we would like to replace the ideal $I$ considered with an arbitrary test module of finite G-dimension and deduce that $R$ is Gorenstein.
For this purpose, we introduce the following category:
$$
\tei(R)=\{M\in \md R\mid\text{all }R\text{-modules }N\text{ with }\Ext^{\gg 0}_{R}(M,N)=0\text{ satisfy }\id_RN<\infty\}.
$$

The theorem below is the main result of this section.
We refer the reader to \cite[V]{Hartshorne}, \cite[A.8]{Larsbook} and \cite[\S1]{Peter} for details of dualizing complexes.

\begin{thm} \label{mainthm1}
Let $R$ be a commutative Noetherian ring (not necessarily local) admitting a dualizing complex.
Then one has $\ttp(R)=\tei(R)$. 
\end{thm}

\begin{proof}
Let $D$ be a dualizing complex of $R$.
Let $M\in \ttp(R)$, and let $X\in \md R$ such that $\Ext_{R}^{\gg 0}(M,X)=0$.
By \cite[A.4.24]{Larsbook} we have an isomorphism in the derived category of $R$:
$$
M\tensor_{R}Y\simeq\RHom(\RHom(M,X),D),
$$
where $Y:=\RHom(X, D)$ is a homologically bounded complex.
Since $\RHom(M,X)$ is homologically bounded, so is $M\tensor_{R}Y$.
Take a projective resolution of $Y$:
$$
(F,\partial)=(\cdots\to F_i\xrightarrow{\partial_i} F_{i-1} \to \cdots \to F_{t+1} \xrightarrow{\partial_{t+1}} F_t \to 0).
$$
As $\Ho_{\gg0}(F)=0=\Ho_{\gg0}(M\otimes_{R}F)$, we can choose an integer $n$ such that the truncation $Q=( \dots \to F_{n+1} \xrightarrow{\partial_{n+1}} F_{n}\to 0)$ of $F$ is a projective resolution of $N:=\coker\partial_{n+1}$ and that $\Tor^{R}_{>0}(M, N)=0$.
Since $M\in \ttp(R)$, this implies $\pd_RN<\infty$, and hence $\pd_RQ<\infty$ as $Q\simeq \Sigma ^{n}N$.
There is a short exact sequence of complexes (\cite[A.1.17]{Larsbook}):
$$
0 \to P \to F \to Q \to 0,
$$
where $P:=(0 \to F_{n-1} \to \dots \to F_{t}\to 0)$.
Since $\pd_RP<\infty$, we have $\pd_RF<\infty$, and hence $\pd_RY<\infty$.
Thus $\id_R\RHom(Y,D)<\infty$.
We have an isomorphism $X\simeq\RHom(Y,D)$ by \cite[V.2.1]{Hartshorne}, which yields $\id_RX<\infty$.
Therefore $M\in\tei(R)$.

Conversely, let $M\in \tei(R)$.
Let $X\in \md R$ such that $\Tor^{R}_{\gg 0}(M,X)=0$.
Similarly to the above, we can prove that $\fd_RX<\infty$; use the isomorphism $\RHom(M\tensor_{R}X,D)\simeq \RHom(M,\RHom(X,D))$ (\cite[A.4.21]{Larsbook}).
We obtain $\pd_RX<\infty$; see \cite[1.6]{Peter}.
\end{proof}

\begin{cor}\label{propmainthm0}
Let $R$ be a commutative Noetherian ring with a dualizing complex.
Assume there exist $M\in \ttp(R)$ and $N\in \md R$ with $\Supp N=\Spec R$ and $\Ext^{\gg 0}_{R}(M,N)=0$.
Then $R$ is Cohen-Macaulay.
If moreover $\pd_RN<\infty$, then $R$ is Gorenstein. 
\end{cor}

\begin{proof}
Since $M\in \tei(R)$ by Theorem \ref{mainthm1}, it holds that $\id_RN<\infty$.
For all $p\in \Spec R$, we have $N_{p} \neq 0$ and $\id_{R_p}N_{p}<\infty$.
The theorem called ``Bass's conjecture'' \cite[9.6.2 and 9.6.4]{BH} yields that $R_{p}$ is Cohen-Macaulay, and so is $R$.
Now assume $\pd_RN<\infty$.
Then $N_{p}$ is a nonzero $R_p$-module of finite projective and injective dimensions for all $p\in \Spec R$.
Hence $R_{p}$ is Gorenstein by \cite[3.1.25]{BH}, and so is $R$.
\end{proof}

In the next corollary, under the hypothesis that the ring considered has a dualizing complex, we obtain a generalization of the result due to Corso--Huneke--Katz--Vasconcelos and Goto--Hayasaka, which accomplishes our main purpose of this section.

\begin{cor}\label{cor1mainthm1}
Let $R$ be a ring with a dualizing complex and $M$ a test module.
\begin{enumerate}[\rm(i)]
\item
If $\GC-dim_RM<\infty$ for some semidualizing module $C$, then $R$ is Cohen-Macaulay.
\item
If $\G-dim_{R}M<\infty$, then $R$ is Gorenstein.
\end{enumerate}
\end{cor}

\begin{proof}
(i) Since $\Hom_R(C,C)\cong R$, we have $\Supp C=\Spec R$.
It is easy to see that $\Ext_{R}^{\gg 0}(M,C)=0$.
By Corollary \ref{propmainthm0}, $R$ is Cohen-Macaulay.

(ii) This follows from Corollary \ref{propmainthm0}.
\end{proof}

The conclusion of Corollary \ref{cor1mainthm1} naturally raises the following question.
We refer to \cite{AGP} for details of the {\em complete intersection dimension} $\CI_R$.

\begin{ques}\label{QCI}
Let $R$ be a local ring.
Let $M$ be a test module with $\CI_RM<\infty$.
Then must $R$ be a complete intersection?
\end{ques}

\noindent
Corollary \ref{propRegularextension} more or less supports affirmativeness of this question, but we do not know the answer.
The difficulty we face here is that we do not know whether the property of being a test module is preserved under local flat extensions, even under completion.

Next we investigate the category $\tei(R)$ when $R$ is Cohen-Macaulay.
Set
$$
\tep(R)=\{M\in \md R\mid \text{all }R\text{-modules }N\text{ with }\Ext^{\gg 0}_{R}(N,M)=0\text{ satisfy }\pd_RN<\infty\}.
$$

\begin{prop} \label{cor7mainthm1}
Let $R$ be a Cohen-Macaulay local ring with a canonical module $\omega$.
Put $(-)^{\dagger}=\Hom_{R}(-,\omega)$.
For each maximal Cohen-Macaulay $R$-module $M$ one has:
$$
M \in \ttp(R)\ \Longleftrightarrow\ M^{\dagger} \in \tep(R).
$$
\end{prop}

\begin{proof}
There are isomorphisms in the derived category of $R$: $\RHom_R(X\tensor_RM,\omega)\cong\RHom_R(X,M^\dag)$ and $\RHom_R(\RHom_R(X,M^\dag),\omega)\cong X\tensor_R\RHom_R(M^\dag,\omega)\cong X\tensor_RM$; see \cite[A.4.21 and A.4.24]{Larsbook}.
These give rise to spectral sequences
\begin{align*}
&{}^1E_2^{p,q}=\Ext_R^p(\Tor_q^R(X,M),\omega) \Rightarrow \Ho^{p+q}=\Ext_R^{p+q}(X,M^\dag)\quad\text{and}\\
&{}^2E_2^{p,q}=\Ext_R^p(\Ext_R^{-q}(X,M^\dag),\omega) \Rightarrow \Ho^{p+q}=\Tor_{-p-q}^R(X,M).
\end{align*}
By using these sequences one can easily deduce the equivalence.
\end{proof}

Recall that a local ring $R$ is called {\em G-regular} (\cite{greg}) if $\G-dim_RM=\pd_RM$ for all $R$-modules $M$.
The above proposition gives a sufficient condition for a local ring to be G-regular in terms of test modules.

\begin{cor}\label{testw}
Let $R$ be a Cohen-Macaulay local ring with a canonical module $\omega$.
If $\omega$ is a test module, then $R$ is G-regular.
\end{cor}

\begin{proof}
Proposition \ref{cor7mainthm1} implies $R\in \tep(R)$.
Let $M$ be an $R$-module.
If $\G-dim_RM<\infty$, then $\Ext_R^{\gg0}(M,R)=0$, and hence $\pd_RM<\infty$.
This shows $\G-dim_RM=\pd_RM$.
\end{proof}

\section{Categorical approaches for nontest modules}

In this section we continue studying the homological properties of test modules, with a special attention to the full subcategory $\NT(R)$ of $\md R$ consisting of nontest modules:
$$
\NT(R)=\ttp(R)^{\rm c}=\{M\in \md R\mid\Tor_{\gg 0}^{R}(M,N)=0 \text{ for some $R$-module } N \notin \fpd(R) \},
$$
where $\fpd(R)$ denotes the full subcategory of $\md R$ consisting of modules of finite projective dimension.
Note that unless $R$ is regular one has:
$$
\NT(R)\supseteq\fpd(R).
$$

We begin with considering closedness of $\NT(R)$ under (finite) direct sums.
First we confirm that it does not always hold:

\begin{eg} \label{notclosed}
Let $k$ be a field and put $R=k[x,y,z]/(x^2, y^2, z^2, yz)$.
Then $R$ is a non-Gorenstein local ring such that the cube of the maximal ideal is zero.
Let $M=R/(x)$ and $N=R/(y,z)$.
Then $M,N \notin \fpd(R)$ and $\Tor_{>0}^{R}(M,N)=0$, hence $M, N \in \NT(R)$.
Suppose $\Tor_{\gg 0}^{R}(M \oplus N,L)=0$ for some $L\in \md R$.
There are exact sequences $0 \to k^{2} \to M \to k \to 0$ and $0 \to k \to N \to k \to 0$.
The first sequence implies $\Tor_{i+1}^{R}(k,L) \cong \Tor_{i}^{R}(k^2,L)$, and hence $\beta^{R}_{i+1}(L) = 2\beta^{R}_{i}(L)$ for $i\gg0$.
Similarly it follows from the second sequence that $\beta^{R}_{i+1}(L) = \beta^{R}_{i}(L)$ for $i\gg0$.
Such equalities of Betti numbers of $L$ can occur only when $\pd_RL<\infty$.
This proves that $M\oplus N$ is a test module, i.e., $M\oplus N \notin \NT(R)$.
\end{eg}

In general we have the following result.

\begin{prop}\label{propdirectsum}
Let $(R,\mathfrak m,k)$ be a non-Gorenstein local ring with $\mathfrak{m}^3=0$.
Let $M$ be a nonfree totally reflexive $R$-module.
Then $M,E_R(k) \in \NT(R)$ and $M\oplus E_{R}(k)\notin \NT(R)$.
\end{prop}

\begin{proof}
Note that $M$ and $E:=E_R(k)$ have infinite projective dimension.
Corollary \ref{testw} implies $E\in \NT(R)$.
Setting $(-)^{\ast}=\Hom(-,R)$ and $(-)^{\vee}=\Hom(-,E)$, we deduce:
$$
M\otimes_R^{\bf L}E
\cong M\otimes_R^{\bf L}\RHom_R(R,E)
\cong\RHom_R(\RHom_R(M,R),E)
\cong M^{\ast\vee}.
$$
Here the second isomorphism follows from \cite[A.4.24]{Larsbook}, and the total reflexivity of $M$ yields the third isomorphism.
Hence we have $\Tor_{>0}^{R}(M,E)=0$.
In particular, $M,E\in \NT(R)$.

Let $L\in \md R$ and assume $\Tor_{\gg 0}^{R}(M\oplus E, L)=0$.
Then $\Tor_{\gg0}^R(M,L)=\Tor_{\gg0}^R(E,L)=0$.
It follows from \cite[2.9]{CSV} that $\cx_RL\leq 1$ and $\cx_RM=1$.
One also sees that $\cx_{R}(M\otimes_{R}L)=\cx_{R}M+\cx_{R}L$.
The complexity of an $R$-module can only be $0$, $1$ or $\infty$ by \cite{Lescot} (see also \cite[1.1]{ChVe}).
Hence $\cx_RL=0$, i.e., $\pd_RL<\infty$.
Thus $M\oplus E \in \ttp(R)$. 
\end{proof}

In fact, closure under direct sums is equivalent to closure under extensions:

\begin{prop}\label{propconj}
Let $R$ be a local ring.
The following are equivalent:
\begin{enumerate}[\rm(i)]
\item
$\NT(R)$ is closed under extensions.
\item
$\NT(R)$ is closed under direct sums.
\end{enumerate}
\end{prop}

\begin{proof}
Clearly, (i) implies (ii).
Assume (ii) holds.
Let $0 \to M \to U \to N \to 0$ be an exact sequence with $M, N\in \NT(R)$.
Then there is an $R$-module $X$ with $\pd_RX=\infty$ such that $\Tor^{R}_{\gg 0}(M\oplus N,X)=0$.
Hence $\Tor^{R}_{\gg 0}(U,X)=0$, and thus $U\in \NT(R)$.
\end{proof}

If $R$ is a nonregular complete intersection, then $\NT(R)$ is closed under direct sums by Proposition \ref{propCI}.
Hence we deduce the following result from Proposition \ref{propconj}.

\begin{cor}\label{12241}
If $\NT(R)$ is closed under direct sums, then $\NT(R)$ is a resolving subcategory of $\md R$.
Thus, $\NT(R)$ is resolving when $R$ is a nonregular complete intersection.
\end{cor}

We state here a conjecture of David A. Jorgensen (cf. Remark \ref{2232} below).

\begin{conj}\label{conjtest}
Let $R$ be a local ring.
Assume $\NT(R)\neq \fpd(R)$.
If $\NT(R)$ is closed under direct sums, then $R$ is a complete intersection of codimension at least two.
\end{conj}

Next we investigate nontest modules in resolving subcategories.
For $M\in\md R$ we denote by $\res M$ the resolving closure of $M$.
The full subcategory of $\md R$ consisting of maximal Cohen-Macaulay modules is denote by $\CM(R)$.

\begin{prop}\label{mainprop1}
Let $R$ be a Henselian local ring, and $\mathcal{X}$ a resolving subcategory of $\md R$.
Suppose there are only finitely many nonisomorphic indecomposable modules in $\mathcal{X} \cap \NT(R)$.
Then $\Ext^{\gg 0}_{R}(M,R)=0$ if and only if $\pd_RM<\infty$ for all $M\in \mathcal{X} \cap \NT(R)$.
In particular, $\G-dim_RM=\pd_RM$.
\end{prop}

\begin{proof}
Take a module $M\in\mathcal{X} \cap \NT(R)$ with $\Ext^{\gg 0}_{R}(M,R)=0$.
Assume $\pd_RM=\infty$.
Given $X\in \md R$, we set $\NT_{X}(R)=\{M\in \md R\mid\Tor_{\gg 0}^{R}(M,X)=0\}$. It is easy to see that $\NT_{X}(R)$ is a resolving subcategory of $\md R$, and
$$
\NT(R)=\textstyle\bigcup_{\pd X=\infty} \NT_{X}(R).
$$
There is an $R$-module $X$ with $\pd_R X=\infty$ and $M\in \NT_{X}(R)$.
Hence $\res M\subseteq \NT_{X}(R) \subseteq \NT(R)$, and we have $\res M\subseteq \mathcal{X} \cap \NT(R)$.
By assumption, there are only finitely many nonisomorphic indecomposable modules in $\res M$ (so $\res M$ is contravariantly finite).
As $\Omega^{d}k\in \ttp(R)$, we have $\Omega^{d}k \notin \res M$.
Hence by \cite[1.4]{Ryo1}, $R$ is Cohen-Macaulay and $\res M=\CM(R)$.
This is a contradiction since $\Omega^{d}k\in \CM(R)$.
Consequently $\pd_RM<\infty$.
\end{proof}

\begin{cor} \label{nothypersurface}
Let $R$ be a Henselian local ring.
Assume that there are only finitely many indecomposable $R$-modules (up to isomorphism) in $\CM(R) \cap \NT(R)$.
\begin{enumerate}[\rm(i)]
\item
If $M\in \CM(R) \cap \NT(R)$, then either $M$ is free or $\G-dim_{R}M=\infty$. 
\item
If $R$ is Gorenstein and nonregular, then $\NT(R)=\fpd(R)$.
\end{enumerate}
\end{cor}

\begin{proof}
The assertion (i) is immediate from Proposition \ref{mainprop1}: take $\mathcal{X}=\CM(R)$.
As to (ii), let $Z\in \NT(R)$.
Then $\Omega^{d}Z \in \CM(R) \cap \NT(R)$, where $d=\dim R$.
By (i), $\Omega^dZ$ is free.
\end{proof}

\begin{rmk}\label{2232}
Recall that $\NT(R)=\fpd(R)$ if $R$ is a hypersurface (Corollary \ref{cor2propCI}(i)).
In view of this fact, it is worth noting that a Henselian Gorenstein ring satisfying the hypotheses of Corollary \ref{nothypersurface} is not necessarily a hypersurface: Let $k$ be a field and let $\displaystyle{R=k[x,y,z]/(x^2-y^2, x^2-z^2, xy, xz, yz)}$.
Then $R$ is an Artinian (hence, Henselian) Gorenstein local ring that is not a hypersurface.
By \cite[3.1(2)]{CSV} we see that $\NT(R) =\fpd(R)$ and that $R$ is the unique indecomposable module in $\CM(R)\cap\NT(R)$.
\end{rmk}

It is not known whether there exist modules $M$ over arbitrary local rings $R$ such that $\pd M=\infty$ and $\Tor^{R}_{\gg 0}(M,M)=0$.
Our next result determines certain conditions, in terms of the category $\NT(R)$, for the existence of such modules.

\begin{prop} \label{mainprop2}
Let $R$ be a local ring.
Assume $\NT(R) \neq \fpd(R)$ and $\NT(R)=\res M$ for some $M\in \md R$.
Then there is $X \in \NT(R)$ with $\pd_RX=\infty$ and $\Tor^{R}_{\gg 0}(X,X)=0$.
\end{prop}

\begin{proof}
If $\pd_RM<\infty$, then $\res M\subseteq\fpd(R)$, hence $\NT(R)=\fpd(R)$.
This contradiction shows $\pd_RM=\infty$.
Since $M\in\NT(R)$, we have $M\in \NT_{X}(R)$ for some module $X$ with $\pd_RX=\infty$.
As $\pd_RM=\infty$, we see that $X\in \NT(R)$.
It holds that $\NT(R)=\res M\subseteq \NT_{X}(R)$, whence $\NT(R)=\NT_{X}(R)$.
Thus $X\in \NT_{X}(R)$, so that $\Tor^{R}_{\gg 0}(X,X)=0$.
\end{proof}

Using the above proposition, we obtain an interesting property of $\NT(R)$:

\begin{cor} \label{cicor}
Let $R$ be a local complete intersection of codimension at least two.
Then $\NT(R)\neq \res M$ for all $R$-modules $M$.
\end{cor}

\begin{proof}
It follows from \cite[5.7]{AGP} that there is an $R$-module $N$ of complexity $1$.
By Proposition \ref{propCI}, $N$ is a nontest module.
Therefore, if $\NT(R)=\res M$ for some $M\in \md R$, then Proposition \ref{mainprop2} implies that there exists $X \in \NT(R)$ with $\pd_RX=\infty$ and $\Tor^{R}_{\gg 0}(X,X)=0$.
Since $R$ is a complete intersection, such an $R$-module $X$ cannot exist by \cite[1.2]{Jo2}.
\end{proof}

\begin{rmk}
The assumption of Corollary \ref{cicor} on the codimension cannot be weakened.
Indeed, let $R$ be a hypersurface (i.e. $\cod R\le1$).
Then $\NT(R)=\fpd(R)$ by Corollary \ref{cor2propCI}(i).
So, if $R$ is reduced of dimension one, then $\NT(R)$ coincides with the resolving closure of the Auslander transpose of $k$ by \cite[2.1]{crspd}.
\end{rmk}

%%%%%%%%%%%%%%%%%%%%%%%%%%%%%%%%%%%%%%%%%%%%%%%%%%%%%%%%%%%%%%%%%%%%%%%
\begin{ac}
This work started around ``Homological Days'' held at the University of Kansas in May, 2011.
We would like to thank the Department of Mathematics of the University of Kansas for hosting the conference.
Our special thanks are due to Petter Andreas Bergh and David A. Jorgensen for valuable comments and suggestions on earlier versions of this paper.
Part of this work was done during the visit of the third author to the University of Missouri in November, 2011, and the visit of the first author to University of Kansas in April, 2012.
We are grateful for kind hospitality of both institutions.
Finally, we thank the referee for valuable comments.
\end{ac}
%%%%%%%%%%%%%%%%%%%%%%%%%%%%%%%%%%%%%%%%%%%%%%%%%%%%%%%%%%%%%%%%%%%%%%%

\end{document}